\documentclass[10pt,a4paper]{amsart}
\usepackage[utf8]{inputenc}
\usepackage{amsmath}
\usepackage{amsfonts}
\usepackage{cite}
\usepackage{xfrac}
\usepackage{amssymb}
\usepackage{mathtools}
\usepackage{amsthm}
\usepackage{graphicx}
\usepackage{tikz}
\usepackage{verbatim}
\usetikzlibrary{positioning}
\usetikzlibrary{arrows}
\usepackage[all]{xy}
\newtheorem{theorem}{Theorem}
\newtheorem{lemma}{Lemma}
\newtheorem{proposition}{Proposition}
\newtheorem{corollary}{Corollary}

\theoremstyle{definition}
\newtheorem{definition}{Definition}
\newtheorem{example}{Example}
\newtheorem{remark}{Remark}
\newcommand{\C}{\mathcal{C}} 
 
\newcommand{\D}{\mathcal{D}}

\newcommand{\N}{\mathbb{N}} 
\newcommand{\Q}{\mathbb{Q}} 
 
\newcommand{\hn}{\hat{N}}
\newcommand{\id}{\text{id}} 
\newcommand{\ordcup}{\oplus}
\newcommand{\nerve}{N}

\begin{document}
\title{(Weak) incidence bialgebras of monoidal categories}
\author{Ulrich Kr\"ahmer}
\author{Lucia Rotheray}
\address{Institut f\"ur Geometrie, Technische Universit\"at Dresden }
\email{ulrich.kraehmer@tu-dresden.de,lucia.rotheray@tu-dresden.de}

\thanks{It is our pleasure to thank Gabriella B\"ohm, Joachim Kock
and Michele Sevegnani for discussions and
suggestions.}

\begin{abstract}
Incidence coalgebras of
categories in the sense of Joni and Rota are studied, specifically cases where a monoidal product on the
category turns these into (weak) bialgebras. 
The overlap with the theory of combinatorial
Hopf algebras and that of Hopf quivers is
discussed, and examples including
trees, skew shapes, Milner's bigraphs and crossed modules are
considered.
\end{abstract}

\maketitle
\tableofcontents

\section{Introduction}

Incidence coalgebras of categories, defined
in \cite{jr}, have been studied in several
areas, notably M\"obius inversion
\cite{le,lm} and combinatorial Hopf algebras (see e.g.
\cite{lr,lf,gr,holt}). The latter notion refers to
Hopf algebras with a vector space basis
indexed by a family of combinatorial objects
(e.g.~graphs or integer partitions), but the
precise definition varies in the literature.
The product and coproduct reflect unions and
compositions of these objects, hence the
underlying coalgebra is (or is closely
related to) an incidence coalgebra of a
category $\C$.

In this case, one might wonder if the
multiplication corresponds to a monoidal
product on $\C$. This was explored by several
authors, see \cite{kgt} for a recent account. 
Here, we characterise two classes of
monoidal categories giving rise to pointed
respectively weak bialgebras:

\begin{theorem}\label{intro1}
If a monoidal product on a M\"obius category
$\C$ has the unique lifting of factorisation property, then
its linearisation turns the incidence coalgbera
of $\C$ into a pointed bialgebra. This is a
Hopf algebra provided that the monoid of
objects is a group.
Similarly, if $\C$
is a locally finite strict 2-group, the monoidal product turns the incidence coalgebra of $\C$ into a weak Hopf algebra.  
\end{theorem}

See the main text for definitions. 

A main goal of this paper is to 
discuss how several well-known examples of
Hopf algebras fit into this picture,
including the Connes-Kreimer Hopf algebra of 
rooted trees and symmetric functions.
Milner's bigraphs, a combinatorial structure
employed in theoretical computer science, are
considered as a new example. We discuss how the Hopf algebraic techniques applied to the problem of renormalisation in physics could lead to new approaches to bigraphical systems, e.g. in studying reaction rules, and it is our hope that there are other parallels to be drawn between bigraphical systems and physical ones, for example Dyson-Schwinger type equations for generating sub-bialgebras. 

We also show that not all combinatorial Hopf algebras (in the above sense) can be described this way, even when the coalgebra structure is the incidence coalgebra of a category. As an illustrative, and perhaps counterintuitive, example we investigate the Hopf quivers of Cibils and
Rosso:

\begin{theorem}
Let $Q=(Q_0,Q_1)$ be a Hopf quiver, $k$ a field and $\cdot:kQ\times kQ\rightarrow kQ$ the multiplication in the associated Hopf algebra $k\C$ as defined in \cite{cr}. Then $\cdot$ is the linear extension of a monoidal product on $Q$ if and only if $Q_1=\emptyset$.
\end{theorem} 

The structure of the paper is as follows:

In Section~\ref{section definitions} we
recall some basic definitions, including that
of an incidence coalgebra of a category,
M\"obius categories and  the unique lifting
of factorisations property. 

In Section~\ref{bialg section} we study
monoidal M\"obius categories whose monoidal
product is an ULF functor. We prove that
these define bialgebras, and discuss the
examples mentioned above. 

In Section~\ref{weak section} 
we recall
the definition of a weak bialgebra, and show
that the monoidal product in a 2-group
satisfies a weak version of the unique
lifting of factorisation property, which
leads to the last statement in
Theorem~\ref{intro1}.

\section{Definitions and notation}\label{section definitions}
Throughout, all categories are assumed to be
small and all monoidal categories to be
strict. We use $\C$ to denote both a category
and its set of morphisms, and $\C(x,y)$ the
subset of morphisms from $x$ to $y$. We
denote the set of objects by $Ob\C$ and the
set of identity morphisms in $\C$ by $Id\C$.
The identity morphism at $x\in Ob\C$ is
written $i_x$. The monoidal product is
denoted $\cdot:\C\times\C\rightarrow \C$. 
\begin{definition}[Decompositions and length]
Given a morphism $f\in\C$ and $n\in\N$, we define the set 
$$
	N_n(f):=\{(a_1,\ldots,a_n)\in\C^{\times n}
	\mid a_1\circ\ldots\circ a_n=f\} 
$$
of $n$-decompositions of $f$, 
and let $\hn_n(f)$ be the non-degenerate
subset, i.e. those decompositions for which
no $a_i$ is an identity morphism. Further, we
set
$$
	\hn(f):=\bigcup_{n\in\N}\hn_n(f),\;\;\ell(f):=\sup\{n\mid\hn_n(f)\neq\emptyset\} 
$$ 
and call $\ell(f)$ the length of $f$.
\end{definition}
Note that by definition, $\nerve_n:=\bigcup_{ f \in \C}
\nerve_n( f)$ is the set of $n$-simplices in the 
nerve of the category $\C$. This leads to the more
topological viewpoint of \cite{kgt}.
The first of the following concepts appears under various names 
in the literature; we follow the terminology of 
Joni and Rota \cite{jr} and Leroux \cite{le}, see
also \cite{lm}. 
\begin{definition}[Locally finite and 
M\"obius categories]
A category $\C$ is called locally finite
respectively M\"obius if
$|N_2(f)|$ respectively 
$|\hn(f)|$ is finite for every $f\in\C$.
\end{definition}
Evidently, a M\"obius category does not contain 
any nontrivial isomorphisms or idempotents.
\begin{lemma}
$\C$ is M\"obius if and only if $\C$ is locally finite and every morphism has finite length.
\end{lemma}
For the proof see \cite{le} or \cite[Proposition~2.8]{lm}.

Note that the sequence of sets
$$
	\C_n := 
	\{  f \in \C \mid
	\ell( f) \le n\}
$$
is a filtration of $\C$ which is exhaustive if and only if every morphism has finite length, in particular when $\C$ is M\"obius. 
The set $\C_0$ contains only identity
morphisms, and $\C$ is called strongly
one-way if $Id\C=\C_0$. This also holds
if $\C$ is M\"obius. The following is
due to Joni and Rota \cite{jr} -- we
consider a variation in which the
coproduct is multiplied by a nonzero
scalar, as this rescaling arises
naturally from the compatibility with
the associative products discussed below:
\begin{theorem}[(Scaled) incidence
coalgebra of a category]\label{modified
path coalg defn}
Let $\C$ be a category, $k$ a field,
$k\C:=\mathrm{span}_{k}\C$ the free
$k$-module on $\C$ and $\lambda \in k$,
$\lambda \neq 0$. Then the following
defines a coassociative counital
coalgebra structure on $k\C$ if and only
if $\C$ is locally finite:
$$
	\Delta ( f):=\frac{1}{\lambda}
	\sum_{( a, b)\in \nerve_2( f)}
	 a\otimes b,\qquad
	\varepsilon ( f):=
	\begin{cases}
	\lambda,\;  f \in Id\C,\\
	0,\; f\notin Id\C 
	\end{cases}
$$
\end{theorem}

We will combine such coalgebra structures on
$k\C$ with the algebra structure defined by a monoidal product on $\C$:
\begin{lemma}\label{alg thm}
Let $(\C,\cdot,1)$ be a monoidal
category, $k$ a field, and
$k\C:=\mathrm{span}_{k}\C$ the free $k$-module on $\C$. Then $(k\C,\cdot,i_1)$ defines an associative, unital $k$-algebra. 
\end{lemma}
To address compatibility of product and
coproduct, we require the following definition:
\begin{definition}[Lifting of factorisations property]\label{ULF def}
For $n\in\N\setminus \{0\}$, a functor $F:\C\rightarrow \D$ has the $n$-to-one lifting of factorisations (nLF) property if the map 
\begin{align*}
N_2(f)\rightarrow N_2(Ff)\\
(a,b)\mapsto (Fa,Fb)
\end{align*}
is an $n$-to-one surjection for all $f\in\C$.  
\end{definition}
This generalises the unique lifting of
factorisations (or ULF) property defined in  e.g.
\cite{kgt,kgt3,lm}. For consistency with the
literature we will use the name ULF rather than 1LF. 
\begin{lemma}\label{ULF and units}
If $F:\C\rightarrow \D$ has the ULF property, then $F$ reflects identities, i.e.  $Ff\in Id\D\Leftrightarrow f\in Id\C$.
\end{lemma}
\begin{proof}
$f\in Id\C\Rightarrow Ff\in Id\D$ holds as
$F$ is a functor. $Ff\in Id\D\Rightarrow f\in
Id\C$ follows from the injectivity of the map $N_2(f)\rightarrow N_2(Ff)$.
\end{proof}
\section{Bialgebras from monoidal categories}\label{bialg section}

In this section we consider the following:
\begin{definition}[Combinatorial category]
A combinatorial category is a monoidal M\"obius category $(\C,\cdot,1)$ whose product $\cdot:\C\times\C\rightarrow \C$ has the ULF property. 
\end{definition}

\begin{theorem}\label{bialgtheorem} 
Let $\C$ be a locally finite category,
$k$ a field, and
$k\C:=\mathrm{span}_{k}\C$ the free
$k$-module on $\C$. We understand $\otimes$ to mean
the tensor product $\otimes_{k}$ and
$(k\C,\Delta,\varepsilon)$ to be the
incidence coalgebra of $\C$.

\begin{enumerate}
\item If $\C$ is M\"obius, the coalgebra $k\C$ is pointed. 
\item If $\C$ is combinatorial, $(k\C,\cdot,i_1,\Delta,\varepsilon)$ defines a $k$-bialgebra. 
\item If $\C$ is combinatorial,
then $k\C$ is a Hopf algebra if and only
if the monoid $(\C_0,\cdot,1)$ 
formed by the objects is a group. 
\end{enumerate}
\end{theorem}

\begin{proof}
\begin{enumerate}
\item As $\C$ is M\"obius, the length filtration is an exhaustive coalgebra filtration, and we apply \cite[Lemma~4.1.2]{der}.
\item The M\"obius condition ensures that $\Delta(i_1)=i_1\otimes i_1$. The ULF property of the monoidal product ensures that $\Delta$ and $\epsilon$ are multiplicative (Lemma \ref{ULF and units} shows this for $\epsilon$).
\item Application of part (1) and \cite[Lemma~7.6.2]{der}. 
\qedhere
\end{enumerate}
\end{proof}

We will now present some examples of
combinatorial categories and the
bialgebras associated to
them as in Theorem~\ref{bialgtheorem}. 
These are
combinatorial bialgebras in the sense
described in the introduction - hence the name
combinatorial category.

\subsection{Example: Relations on monoids}\label{relations}
We begin with a simple class of examples to
illustrate the definitions.

Let $\preccurlyeq$ be a reflexive and
transitive relation on a monoid $M$ for which 
\begin{equation}\label{precrel}
	x \preccurlyeq y,
	z \preccurlyeq t \Rightarrow
	x \cdot z \preccurlyeq y \cdot t
\end{equation}
holds.
Then 
$$
	\C_M:=\{(x,y) \in M \times M \mid
	x \preccurlyeq y\}
$$
is the morphism set of a monoidal category with 
object set $M$ and $(x,y)$ the (unique) morphism $y 
\rightarrow x$. The operations 
$\circ,\cdot$ are given by
$$
	(x,y) \circ (y,z) := 
	(x,z),\quad
	(x,y) \cdot (z,t) :=
	(x \cdot z,y \cdot t).
$$	
The identity morphisms are given by
$
	i_{x} = 
	(x,x).
$ 
By definition, we have:
\begin{proposition}
\begin{enumerate}
\item The category $\C_M$ is locally finite
if and only if all intervals 
$
	[x,y] := \{z \in M \mid 
	x \preccurlyeq z \preccurlyeq y\}$
%, $x \preccurlyeq y$,  
are finite. 
\item The category 
$\C_M$ is M\"obius if and only if
given $x \preccurlyeq y \in M$ there is 
$\ell(x,y)  \in \mathbb{N} $ such that there
is no chain $x=z_1 \preccurlyeq z_2 \preccurlyeq
\ldots \preccurlyeq z_l=y$ with $z_i \neq
z_{i+1}$ and $l > \ell(x,y)$. 
\item The monoidal structure is
ULF if
and only if it defines for all $x \preccurlyeq y,z
\preccurlyeq t \in M$
a bijection
$
	[x,y] \times [z,t] \rightarrow
	[x \cdot z,y \cdot t].
$
\end{enumerate}
\end{proposition}

Note that if $\C_M$ is M\"obius, $\preccurlyeq$ is a
partial ordering. 
\begin{remark}\label{remark poset}
Let $(M,\leq)$ be a monoid with a
partial ordering satisfying (\ref{precrel}).
 On the set $\C_M$ of morphisms define $(x,x)\simeq (1,1)\forall x\in M$. Taking the quotient by the corresponding bialgebra ideal removes non-unit group-like elements and produces a Hopf algebra which is as close as possible to the original bialgebra. 

Various authors, for example \cite{jr,bmsw}, consider Hopf algebras of intervals in posets and impose exactly this condition to obtain Hopf algebras rather than bialgebras. Rather than starting with a monoid and imposing a partial order they start with a grading poset and define a product on its intervals (morphisms for us), which induces a monoid structure on the original set, compatible with the grading. 
\end{remark}
%\begin{example}
%If $\preccurlyeq$ is
%equality, then the category is combinatorial 
%irrespective of the monoid structure, and the
%path bialgebra is the monoid bialgebra $k M$.
%\end{example}

%\begin{example}
In particular, when $\preccurlyeq$ is an
equivalence relation, then the 
interval $[x,y]$ is simply the equivalence
class of $x$ (and of $y$), and
(\ref{precrel}) demands that the monoid
structure descends to the quotient 
$M/\preccurlyeq$. The category $\C_M$ is 
locally finite if and only if all equivalence
classes are finite, but it is not M\"obius
unless $\preccurlyeq $ is $=$, as any 
$(x,y)$ with $x \preccurlyeq y$ but 
$x \neq y$ is a nontrivial ismorphism.
As a coalgebra, we have
$$
	k\C_M \cong 
	\bigoplus_{x \in M/\preccurlyeq} 
	M_{|x|} (k),
$$
where $M_n(k)$ is the matrix coalgebra of size $n
\times n$, so $k \C_M$ is cosemisimple. The
monoidal structure is ULF if
and only if $M/\preccurlyeq \rightarrow 
(\mathbb{N},\cdot)$, $x \mapsto |x|$ is a map of
monoids. 

\begin{example}\label{monex}
As an example of a combinatorial bialgebra arising
in this way, consider 
$M=\langle S \rangle $, the free monoid 
on
$S=\{x,y\}$, in which two words are equivalent if they
have the same length. Then $\C_M = \langle S
\times S \rangle $ and 
$k \langle  S \times S \rangle  $ is the free algebra on
four generators $\alpha=(x,x), \beta=(x,y),\gamma=(y,x),
\delta=(y,y)$ whose 
coproduct is given by 
$$
	\Delta (\alpha) = \alpha \otimes \alpha
	+ \beta \otimes \gamma,\quad
	\Delta (\beta) = \alpha \otimes \beta + 
	\beta \otimes \delta,
$$
$$
	\Delta (\gamma) = \gamma \otimes \alpha+
	\delta \otimes \gamma,\quad
	\Delta (\delta) = \gamma \otimes \beta + 
	\delta \otimes \delta.
$$
In particular, it is not a Hopf algebra. 
\end{example}

\subsection{Example: Skew shapes and symmetric functions}\label{section shapes}
We now elaborate further on a specific
case of the previous example and realise the
Hopf algebra $\Lambda$ of symmetric functions 
(see e.g. \cite[Section~2.1]{gr}) as a
quotient of an incidence bialgebra. 
 
\begin{definition}[Path]
For $n\in\N$, a path of length $n$ is a word of length $n$ in the alphabet $\{0,1\}$. For a path $p=p_1\ldots p_n$ we define its height $h(p)=\sum_i p_i$ and width $w(p)=n-h(p)$. 
\end{definition}

\begin{definition}[Skew shape]

For two paths $q,p$ of equal height and width we define:
\begin{align*}
q\leq p\Leftrightarrow \sum_{j=1}^{i}p_j\geq
\sum_{j=1}^{i}q_j \forall i=1,\ldots,h+w\\
q< p\Leftrightarrow \sum_{j=1}^{i}p_j >
\sum_{j=1}^{i}q_j \forall i=1,\ldots,h+w-1.
\end{align*}

A skew shape is a pair $(q,p)$ of paths with $q\leq p$. If $q<p$, the shape is called connected.
\end{definition}
\begin{definition}[Category of skew shapes]
We consider a discrete category $\C_S$ with $Ob\C_S$ the set of all paths and $\C_S(q,p)=\{(q,p)\}$ if $q\leq p$ and $\emptyset$ otherwise. The composition and product are defined by 
\begin{align*}
(q,p)\circ (r,q)&=(r,p)\\
p\cdot p'&=(p_1,\ldots,p_n,p'_1,\ldots,p'_m)\\
(q,p)\cdot (q',p')&=(q\cdot q',p\cdot p')
\end{align*}
\end{definition}
We depict a path in the
plane by starting at the origin and 
drawing each 0 and 1 as a step of
unit length in the direction right
respectively up. A skew shape $\mu=(q,p)$ is
drawn as the area bounded by $q$ and $p$. %The squares in this grid are the blocks of $\mu$. The length $\ell(\mu)$ is the area contained, i.e. the number of blocks
\begin{example}
Consider the paths
$$
r=00101,\;\;
q=01010,\;\;
p=10100
$$
Then:
\begin{align*}
(q,p)\circ (r,q)&=(r,p)\\ \raisebox{-.5\height}{\begin{tikzpicture}[scale=0.5, line width=1pt]
  \draw (0,0) grid (1,1);
  \draw (1,1) grid (2,2);
   \draw (2,2) grid (3,2);
  \end{tikzpicture}}\; \circ\raisebox{-.5\height}{\begin{tikzpicture}[scale=0.5, line width=1pt]
  \draw (0,0) grid (1,0);
    \draw (1,0) grid (2,1);
     \draw (2,1) grid (3,2);
  \end{tikzpicture}}&=\raisebox{-.5\height}{\begin{tikzpicture}[scale=0.5, line width=1pt]
  \draw (0,0) grid (2,1);
  \draw (1,1) grid (3,2);
   \draw (2,2) grid (3,2);
  \end{tikzpicture}}\\
(q,p)\cdot (r,q)&=(q\cdot r,p\cdot q)\\  \raisebox{-.5\height}{\begin{tikzpicture}[scale=0.5, line width=1pt]
  \draw (0,0) grid (1,1);
  \draw (1,1) grid (2,2);
   \draw (2,2) grid (3,2);
  \end{tikzpicture}}\; \cdot\raisebox{-.5\height}{\begin{tikzpicture}[scale=0.5, line width=1pt]
  \draw (0,0) grid (1,0);
    \draw (1,0) grid (2,1);
     \draw (2,1) grid (3,2);
  \end{tikzpicture}}&= \raisebox{-.5\height}{\begin{tikzpicture}[scale=0.5, line width=1pt]
  \draw (0,0) grid (1,1);
  \draw (1,1) grid (2,2);
   \draw (2,2) grid (3,2);
   \draw (3,2) grid (4,2);
    \draw (4,2) grid (5,3);
     \draw (5,3) grid (6,4);
  \end{tikzpicture}}
\end{align*}
Of these, only $(r,p)$ is connected. $(q,p)$ and $(r,q)$ each have three connected components and $(q\cdot r,p\cdot q)$ has six. %We also have $\ell(q,p)=\ell(r,q)=2$ and $\ell(r,p)=\ell(q\cdot r,p\cdot q)=4$.
\end{example}
We denote by $\Gamma:=\{(q,p)\in \C_S\mid q<p\}$ the subset of connected skew shapes and $\Gamma_0:=\{0,1\}\subset\Gamma$. $\Gamma_0$ and $\Gamma$ freely generate the monoids
$(Id\C_S,\cdot)$ and $(\C_S,\cdot)$ respectively. Hence we can write any shape $(q,p)\in \C_S$ as a product $(p,q)=\gamma_1\cdots\gamma_m$, with all $\gamma_i\in\Gamma$.

For any $\gamma\in\Gamma\setminus\Gamma_0$ let $s_{\gamma}$ denote the (skew) Schur function on $\gamma$  \cite[Definitions~2.8,2.19]{gr} and for $\gamma\in\Gamma_0$, set $s_{\gamma}=1$. We define a map
\begin{align*}
\phi:k\C_S&\rightarrow k\Lambda\\
(p,q)=\gamma1\cdots \gamma&\mapsto \prod_{i=1}^m s_{\gamma} 
\end{align*}
\begin{lemma}
$\phi$ as defined above is a surjective bialgebra homomorphism.
\end{lemma}

Hence $\Lambda$ can be seen as a quotient of an incidence bialgebra.

\begin{remark}[Other quotients of $k\C_S$]
The Hopf algebra of intervals in Young's lattice  \cite{bmsw}
 and the commutative Hopf algebra of skew
shapes \cite{ky} are similarly quotients of $k\C_S$. 

\end{remark}

\subsection{Example: Rooted forests and the Connes-Kreimer Hopf algebra}
This bialgebra was also considered in \cite{dm,kockdse,kaufmannward}, our notation is closest to that of \cite{kockdse}. For any graph $G$, let $V(G)$, $E(G)$ and $C(G)$ denote its sets of vertices, edges and connected components respectively.
\begin{definition}[Operadic planar rooted forest]
A planar operadic rooted forest consists of:
\begin{enumerate}
\item A cycle-free finite graph $f$ in which every connected component $C_i\in C(f)$ has a distinguished element $r_i\in V(f)$ of degree one called its root.
\item A colouring $c:V(f)\rightarrow \{w,b\}$ such that $c(v)=w\Rightarrow deg(v)=1$ and $c(r_i)=w$ for every root.
\item Total orderings on the sets $R(f)$ of roots, $L(f):=c^{-1}(w)\setminus R(f)$ of (white) leaves and $c^{-1}(b)=V(f)\setminus (L(f)\cup R(f))$ of internal vertices satisfying $v\in C_i,w\in C_j, v<w\Leftrightarrow r_i<r_j$.
\end{enumerate}
\end{definition}

\begin{definition}[Category of operadic planar rooted forests]
We consider the PROP $\C_{RF}$ where $\C_{RF}(m,n)$ the set of operadic planar rooted forests with $n$ roots and $m$ white leaves for all $m,n\in\mathbb{N}$. For $f_1\in\C_{RF}(m,n)$ and $f_2\in\C_{RF}(p,m)$, $f_1\circ f_2$ is the forest obtained by matching the roots of $f_2$ to the leaves of $f_1$ as dictated by ordering, then deleting these vertices and merging their incident edges. The monoidal product $f_1\cdot f_2$ is the ordered sum (non-commutative disjoint union), in which all vertices (previously) in $f_2$ are greater than all vertices (previously) in $f_1$. 
\end{definition}
We will draw trees with the root at the top,  and not indicate the total orderings explicitly (we consider leaves and roots  as ordered from left to right). 
\begin{example}\label{treeexample}

\begin{align*}
&f=\raisebox{-.5\height}{\includegraphics[scale=.7]{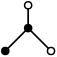}}
&f\circ g=\raisebox{-.5\height}{\includegraphics[scale=.7]{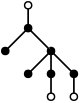}}
\\
&g=\raisebox{-.5\height}{\includegraphics[scale=.7]{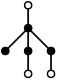}}
&g\circ (f\cdot f)=\raisebox{-.5\height}{\includegraphics[scale=.45]{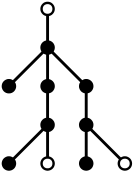}}
\end{align*}
\end{example}

The category $\C_{RF}$ is combinatorial and so by Theorem~\ref{bialgtheorem} defines a bialgebra $\Q\C_{RF}$. This bialgebra, also considered explicitly in \cite{dm,kockdse} is made implicit use of in \cite{MSc,bk} and is closely related to the Hopf algebras of trees defined in  \cite{lf,bk}. In fact, we can view these Hopf algebras of trees as quotients of $\Q\C_{RF}$.
\begin{definition}[Core, core equivalence]
For any forest $f$, let $core(f)$ denote the (non-operadic) rooted forest obtained by deleting $R(f),L(f)$ and all incident edges. The vertices originally adjacent to the roots in $f$ become the roots in $core(f)$. Two forests $f,g$ are core equivalent iff $core(f)=core(g)$.
\end{definition}
In particular note that for any $i_n\in Id\C_{RF}$, $core(i_n)=i_0$, the empty tree. Let $I\triangleleft \Q\C_{RF}$ denote the ideal generated by $\{f-g\mid core(f)=core(g)\}$ in $k\C$.
\begin{lemma}\label{HCK lemma}
$I$ is a bialgebra ideal and the quotient $\Q\C_{RF}/I$ is isomorphic to the planar (i.e. non-commutative) Connes-Kreimer Hopf algebra $H$ of rooted trees \cite{acdk,lf}.
\end{lemma}

Core equivalence and the Connes-Krimer quotient Hopf algebra are also discussed in \cite{kockdse}. Similarly, it is observed in \cite[Section~4.4.2]{kaufmannward} that in general the quotient of any $k\C$ by the bialgebra ideal generated by all $i_x-i_1$ (which is contained in $I$) is a Hopf algebra.

\subsection{Example: A toy model for bigraphs}
Bigraphs \cite{rm1,cs} are a combinatorial
structure originally developed 
in theoretical computer science to model mobile
computation. Bigraphs as defined by Milner form the morphisms of a monoidal pre-category \cite[Section~2.2-2.3]{rm1}, i.e. a monoidal category in which not all compositions or monoidal products are defined.
Here we consider a restricted definition of bigraphs, which define a combinatorial category. This toy model admits the basic bigraph operations (composition, reactions and reductions), but does not include all features. For example, we only consider private names for bigraphs, and essentially work in the support quotient of Milner's pre-category. 
\begin{definition}[Place graph]
A place graph 
consists of finite 
disjoint sets 
$R,S,V$, total orders on $R$ and $S$, and a map 
$
	prnt \colon 
	V \cup S 
	\rightarrow 
	V \cup R
$ 
such that:
\begin{itemize}
\item $R \subseteq \mathrm{im}\, prnt$
\item for all $v \in V \cup S$, $ prnt^\infty(v) \in R$
\item $ prnt^\infty(v):S\rightarrow R$ is non-decreasing
\end{itemize}

\end{definition}
The elements of $R,S$ and $V$ are called roots, sites and vertices respectively. 
Let $f_i=(R_i,S_i,V_i,prnt_i)$,
$i=1,2$ be place graphs with 
$R_1=S_2$ and all other involved sets
pairwise disjoint. The
composition $f_2 \circ f_1$ is defined as 
$(R_2,S_1,V_1 \ordcup V_2,prnt)$, 
where $prnt(v):=prnt_i(v)$ for 
$v \in V_i,S_i$ except when 
$prnt_1(v) \in R_1=S_2$, in which 
case $prnt(v):=prnt_2(prnt_1(v))$.

\begin{definition}[Link graph]
A link graph consists of 
finite disjoint sets $P,X,Y$ with $X,Y$ totally ordered
and an equivalence relation $\sim$
on $P \cup X \cup Y$. Each equivalence class must contain at least two elements and cannot be entirely contained in $X$ or $Y$.
\end{definition}
The elements of $P,X$ and $Y$ are called ports, inner names and outer names respectively. 
If 
$f_i=(P_i,X_i,Y_i,\sim_i)$, $i=1,2$, are link
graphs with $Y_1=X_2$, then 
$$f_2 \circ f_1:=(P_1 \ordcup P_2,
X_1,Y_2,\sim)$$ where $\sim$ is generated by 
$\sim_1$ and $\sim_2$ (including  
$x \sim y$ if $x \in X_1 \cup P_1,y \in P_2 \cup
Y_2$, and $x \sim_1 z \sim_2 y$ for some  
$z \in Y_1=X_2$).
\begin{definition}[Bigraph]
A bigraph
$g=(g_{p},g_{l},\rho)$ consists of
\begin{enumerate}
\item A place graph 
$g_{p}=(R,S,V,prnt)$.
\item A link graph $g_{l}=(P,X,Y,\sim)$.
\item A map $\rho \colon P \rightarrow V$
assigning a vertex to each port.
\end{enumerate}
The pairs $( |S|,|X|)$ and $(|R|,|Y|)$ are called the inner and outer interfaces of $g$. 
\end{definition}
It is often easier to represent a bigraph by a drawing than to define it formally. To do this, we represent roots and sites by boxes, vertices by circles and ports by dots. The parent map is indicated by nesting, the map $\rho$ by placing a port directly on its associated vertex and the equivalence relation $\sim$ by lines joining equivalent elements. 
\begin{example} 
The bigraph with data
\begin{align*}
 V=\{v_0\},\;\; R=\{r_0\}\;\;S=\{s_0,s_1\}\\
 prnt(s_0)=prnt(v_0)=r_0,\;prnt(s_1)=v_0 \\
  X=\{x_0,x_1\},Y=\emptyset P=\{p_0,p\}\\
 x_0\sim p_0,\;x_1\sim p_1\\
 \rho(p_0)=\rho(p_1)=v_0
\end{align*}
will be drawn as:
$$\raisebox{-.5\height}{\includegraphics[scale=.5]{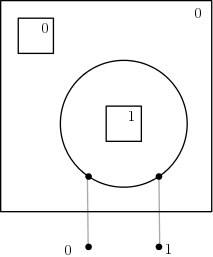}}\;.$$
\end{example}
\begin{definition}[Category of bigraphs]
We consider a category $\C_B$ with $Ob\C_B=\N\times \N$ and $\C_B(( m,x),( n,y))$ the set of all bigraphs with inner and outer interfaces $( m,x)$ and $( n,y)$ respectively. 
The composition of two bigraphs $g=(g_{p},g_{l},\rho_g)$ and $f=(f_{p},f_{l},\rho_f)$ is given componentwise: $$ g\circ f=(g_{p}\circ f_{p},g_{l}\circ f_{l},\rho_g\oplus \rho_f)$$
if the inner interface of $g$ equals the outer
interface of $f$.
The product of $g$ and $f$ is given componentwise by ordered unions (as for trees): $$fg=(f_{p}\oplus g_{p},f_{l}\oplus g_{l},\rho_g\oplus \rho_f).$$
\end{definition}
We denote the identity at the interface $\langle m,x\rangle$ by $i_{m,x}$. The monoid $(Id\C_B,\cdot)$ is freely generated by the set $\{i_{0,1},i_{1,0}\}=\Bigg\{\raisebox{-.3\height}{\includegraphics[scale=.3]{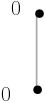}},\;\;\raisebox{-.3\height}{\includegraphics[scale=.25]{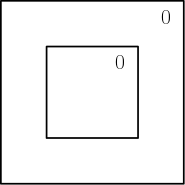}}\Bigg\}$.
\begin{example}[Composition and product]
Consider the bigraphs
\begin{align*}
f=\raisebox{-.5\height}{\includegraphics[scale=.25]{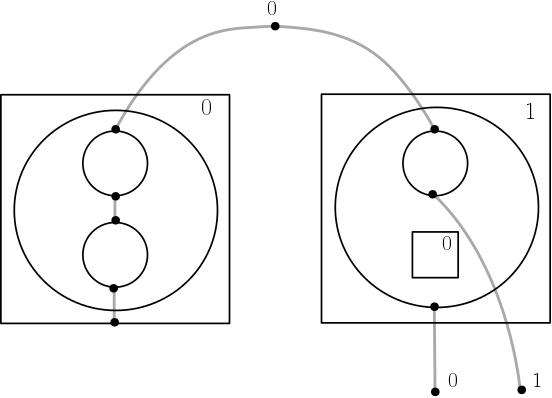}}\;,\;\; g=\raisebox{-.45\height}{\includegraphics[scale=.3]{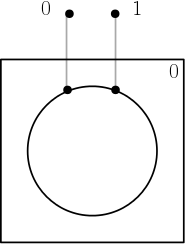}}\\
f'=\raisebox{-.5\height}{\includegraphics[scale=.3]{j.png}}\;,\;\; g'=\raisebox{-.45\height}{\includegraphics[scale=.25]{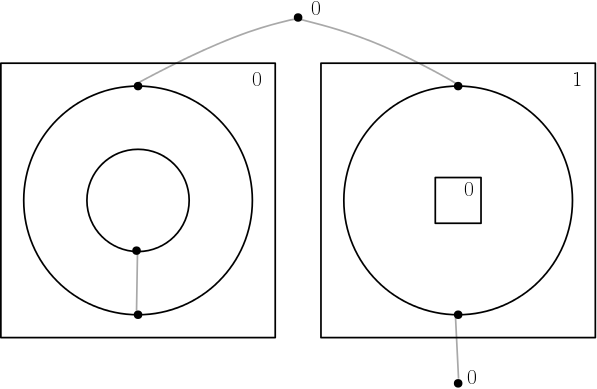}}\;.
\end{align*}
Then
\begin{align*}
f\circ g&=\raisebox{-.5\height}{\includegraphics[scale=.25]{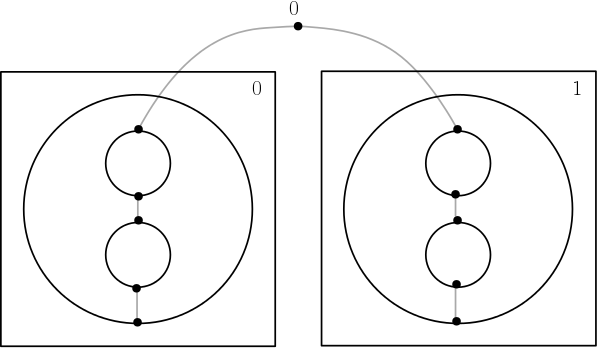}}\;,\\
f'\cdot g'&=\raisebox{-.5\height}{\includegraphics[scale=.25]{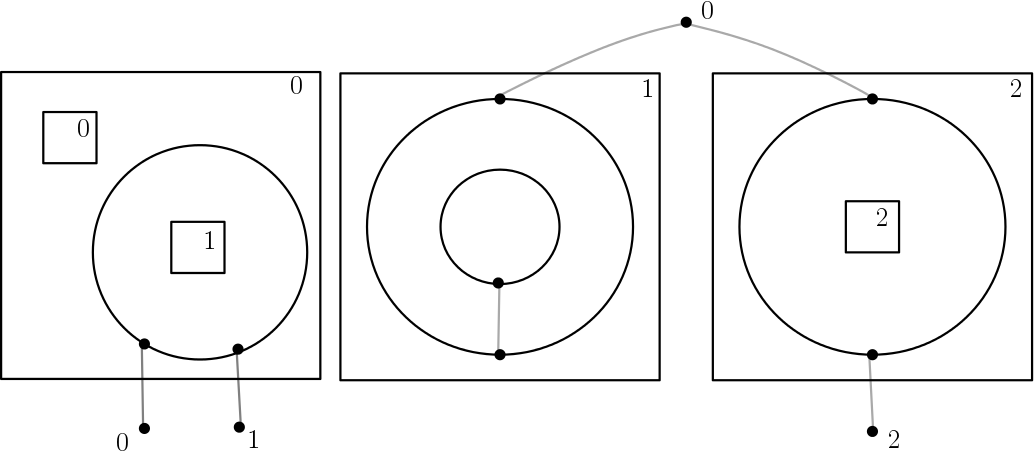}}\;,\\
g'\cdot f'&=\raisebox{-.5\height}{\includegraphics[scale=.25]{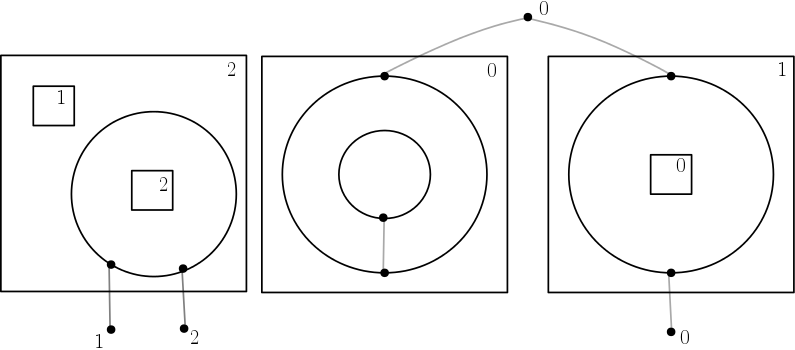}}\;.
\end{align*}
This category is combinatorial and so by Theorem~\ref{bialgtheorem} there exists an incidence bialgebra $k\C_B$.
\end{example}
\begin{example}[Coproduct in $k\C_B$]
Consider the bigraph $f'$ from the previous example. 
$$\Delta(f')=f'\otimes i_{2,2}+i_{1,0}\otimes f'+ \raisebox{-.35\height}{\includegraphics[scale=.3]{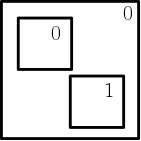}}\otimes\raisebox{-.5\height}{\includegraphics[scale=.3]{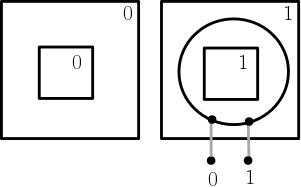}}  $$
\end{example}

We feel that the Hopf algebraic techniques
used to study $\C_{RT}$ and
$H_{CK}$ could provide new approaches to
studying bigraphical systems. In particular,
one possible application is in the study of
reaction rules. Put simply, a reaction rule
is a map $r:\C_B\rightarrow \C_B$ which
removes a certain subset of bigraphs by
mapping them to simpler ones. These rules are
not necessarily compatible with the
composition, i.e.~in general, we have 
$r(g\circ f) \neq r(g)\circ r(f)$.

\begin{example}
Consider the reaction rule for a version of $\C_B$ in which vertices can be labelled by $\{A,B,C\}$: $$\raisebox{-.5\height}{\includegraphics[scale=.3]{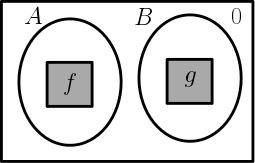}}\;\;\mapsto \;\; \raisebox{-.5\height}{\includegraphics[scale=.3]{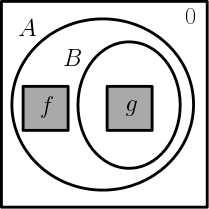}},$$ where the grey boxes $f$,$g$ denote any substructure. This rule demands that the $A$ and $B$ vertices in the initial graph are siblings (i.e. $prnt^n(A)\in R$ and $prnt^n(B)\in R$ for the same $n$). This means that $r$ maps the graphs $$\raisebox{-.5\height}{\includegraphics[scale=.3]{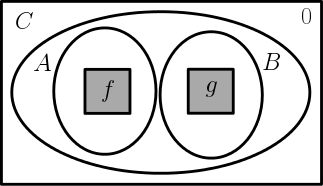}}\;,\;\;\raisebox{-.5\height}{\includegraphics[scale=.3]{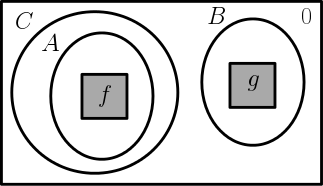}},$$ to $$\raisebox{-.5\height}{\includegraphics[scale=.3]{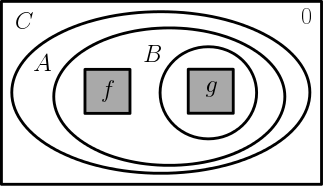}}\;,\;\;\raisebox{-.5\height}{\includegraphics[scale=.3]{react3.png}}.$$ The second graph shows no reaction even though the bigraph does contain the relevant subgraph. 

The coproduct allows us to see such blocked reactions. For our example, define the new map 
$$r':=(\id\otimes rM\circ)\Delta:k\C_B\rightarrow k\C_B\otimes k\C_B,$$ 
where 
$$M:=\raisebox{-.5\height}{\includegraphics[scale=.25]{i10.png}}+\raisebox{-.5\height}{\includegraphics[scale=.3]{02.png}} +\raisebox{-.5\height}{\includegraphics[scale=.3]{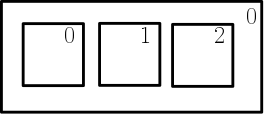}} + \ldots.$$ 
Applying this, we have:
\begin{align*}
r'(\raisebox{-.5\height}{\includegraphics[scale=.3]{react3.png}})
&=\ldots + \raisebox{-.5\height}{\includegraphics[scale=.3]{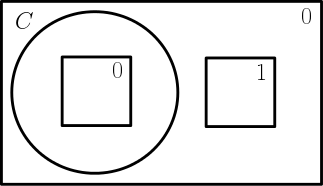}}\otimes r\left(M_2\circ\raisebox{-.5\height}{\includegraphics[scale=.3]{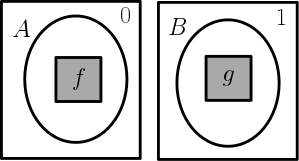}}\right)+\ldots)\\
&=\raisebox{-.5\height}{\includegraphics[scale=.3]{react6.png}}\otimes\ldots+\raisebox{-.5\height}{\includegraphics[scale=.3]{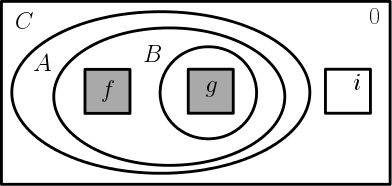}} +\ldots
\end{align*}
where $i\in \N$ is determined by the interfaces of $f$ and $g$.
\end{example}

\subsection{A notable non-example: Hopf quivers}\label{quiver section}

We now consider an example of an incidence coalgebra which admits a multiplication making it a Hopf algebra, but is not (generally) an example of our construction.

Let $Q=(Q_{0},Q_{1})$ be a quiver, $k$ be a field,
and $\C_{Q}$ be the category of all paths in $Q$, 
which is locally finite. By definition, the
coalgebra $k\C_Q$ is the path coalgebra $kQ$ of
the quiver.

Cibils and Rosso \cite{cr} (see also \cite{clw}) 
showed that a graded Hopf algebra structure on $kQ$ 
endows $Q_0$ with a group and $kQ_1$ with a
$kQ_0$-Hopf bimodule structure. This in turn gives
rise to what Cibils and 
Rosso call ramification data.

\begin{definition}[Ramification data]\label{ramdef}
Let $G$ be a group. By ramification data for $G$,
we mean a sum $r=\sum_{C\in\mathfrak{C}}r_{C}C$ for 
$\mathfrak{C}$ the set of conjugacy classes of $G$
and all $r_{C}\in \mathbb{N}$.
\end{definition}

\begin{definition}[Hopf quiver]\label{hopfquiverdef}
Let $G$ be a group and $r$ some ramification data
for $G$. The quiver with $Q_{0}=G$ and $r_{C}$
arrows from $g$ to $cg$ for each $g,c \in G$, where
$C$ is the conjugacy class containing $c$, is
called the Hopf quiver determined by 
$(G,r)$.
\end{definition}
\begin{theorem}
The path coalgebra $kQ$ of a quiver $Q$ admits graded Hopf algebra structures if and only if $Q$ is a Hopf quiver.
\end{theorem}

To be more precise, for a Hopf quiver $Q$, 
there is a canonical Hopf
algebra structure on $kQ$, with 
$kQ_1 \cong \bigoplus_{C \in \mathfrak{C}} 
r_C kC \otimes kQ_0$ as vector space. The
$kQ_0$-bimodule structure is given by 
$$
	x(c \otimes g)y = xcx^{-1} \otimes 
	xgy
$$ 
and the $kQ_0$-bicomodule structure is given by
$$
	c \otimes g \mapsto 
	cg \otimes (c \otimes g) \otimes g.  
$$
The group structure on $Q_0$ 
and the 
$kQ_0$-bimodule structure on $kQ_1$ define the algebra
structure on $kQ$ in lowest degrees, which is
extended universally to all of $kQ$, see 
\cite[Theorem~3.8]{cr}. However, given
a Hopf quiver, there are in general also 
Hopf algebra structures on $kQ$ compatible with 
$r$ that are different from the canonical one. 
The choice is in the $kQ_0$-bimodule
structure of $kQ_1$ which amounts to a choice of
an $n_C$-dimensional representation of the
centraliser of an element $c \in C$. The extension
of the product to paths is then unique.

Huang and Torecillas \cite{ht} proved that a
quiver path coalgebra $kQ$ always admits graded
bialgebra structures. The results of Green and
Solberg \cite{gs} are also closely related, but
different in that they study path algebras rather
than path coalgebras. Here we focus on the Hopf
algebra setting as considered in \cite{cr}. 

Our key aim is to stress that the product in
a quiver Hopf algebra $kQ$ is not a linear
extension of a monoidal product on $\C_Q$ unless 
$Q_1=\emptyset$. To do this, we classify all 
monoidal structures on path categories of quivers
whose vertices form a group under $\cdot\,$:

\begin{lemma}\label{quivercatlemma}
Assume that the path category $\C_Q$ of a
quiver is monoidal such that $(Q_{0},\cdot)$
forms a group. Then:
\begin{enumerate}
\item The monoidal product defines commuting left and
right $Q_0$-actions on $Q_1$. 
\item The path length $\ell$ is a grading
with respect to both $\cdot$ and $\circ$. 
\item Either $Q_1$ is empty, or
there exists an element $z \in Z(Q_{0})$ such
that $Q_1$ contains for each $a \in Q_0$ 
exactly one arrow $f_a \colon a \rightarrow z \cdot a$.
\end{enumerate}
\end{lemma}

\begin{proof}
If $Q_1 = \emptyset$ there is nothing to
prove, so we assume $Q_1 \neq \emptyset$. 

1.~To begin with, we prove that for any 
identity morphism $i_a$, $a \in Q_0$, and any 
arrow $f \colon b \rightarrow c 
\in Q_1$, we have 
$i_a \cdot f,f \cdot i_a \in Q_1$. Indeed, 
we have
$$
	i_a \cdot f = g_1 \circ \ldots 
	\circ g_n
$$
for unique arrows $g_i \in Q_1$, and then
$$
	f = i_{a^{-1}} \cdot 
	(g_1 \circ \ldots 
	\circ g_n) =
	(i_{a^{-1}} \cdot g_1) \circ 
	\ldots \circ  
	(i_{a^{-1}} \cdot g_n)      
$$
where we used that $i_a \cdot 
i_{a^{-1}} = i_{a \cdot a^{-1}} =
i_1$ and that $\cdot$ is a monoidal product. 
It is impossible that $i_{a^{-1}} \cdot g_i =
i_b$ for some $b \in Q_0$, as we would then have
$g_i = i_a \cdot i_b=i_{a \cdot b}$. 
So the right hand side of the above equation 
is a path of length at least $n$ while the left
hand side has length $1$, hence $n=1$.  
Analogously, one proves $f \cdot i_a \in
Q_1$. That $\cdot$ defines commuting actions 
is immediate.  
 
2.~For any arrow $f$, 
let $s(f), t(f)$ denote its source and target vertices
respectively.
As $\cdot$ is a monoidal product, we have for any two arrows $f,g$: 
\begin{equation}\label{entscheidend}
	f\cdot g =
	(i_{t(f)}\cdot g)\circ(f\cdot i_{s(g)})=
	(f \cdot i_{t(g)}) \circ 
	(i_{s(f)} \cdot g). 
\end{equation}
By what has been shown already, this is a 
path of length 2. Continuing inductively, one
proves that $\ell(f \cdot g)=\ell(f)+\ell(g)$ 
for all paths $f,g \in \C_Q$.  

3.~We also conclude from (\ref{entscheidend})
that  
$$
	s(f) \cdot t(g)=t(f) \cdot s(g) 
	\Rightarrow 
	\;
	t(g) \cdot s(g)^{-1} =
	s(f)^{-1}\cdot t(f)
	=: z \in Q_0 
$$
is constant (independent of $f$). So  
$$
	t(f)=s(f) \cdot z=z \cdot s(f)
$$
and any arrow $f \in Q_1$ with 
source $a$ has the same target
$z \cdot a=a \cdot z$. 

Given any arrow $f \colon a\rightarrow a \cdot z$ and 
$b \in Q_0$, there exists  
$i_b \cdot f \colon b \cdot a \rightarrow z \cdot
b \cdot a$. 
This means that the same number of arrows go out
of each vertex and that $z$ is in the
centre of $Q_0$. 

Finally, assume there are two arrows 
$f,g$ with source $1$ (the unit element 
of $Q_0$) and consider again
(\ref{entscheidend}):
$$
	f\cdot g =
	(i_{z}\cdot g)\circ f=
	(f \cdot i_{z}) \circ 
	g. 
$$
We deduce that $f=g$.
\end{proof}

If $Q_1 = \emptyset$, then 
$\C_Q=\{i_a: a\in Q_0\}$ is just a group. The
monoidal prouct is the group multiplication,
and this is ULF. 
The Hopf algebra $k\C_Q$ is the group 
algebra of $Q_0$ and is combinatorial. 
  
If $Q_1 \neq \emptyset$, there are 
two sub-cases: $z=1$ and $z \neq 1$. In the first
case, each vertex $a\in Q_0$ has a unique
arrow $f_a \colon a \rightarrow a$. 
In the second case, each vertex $a$ has one incoming
arrow $f_{z^{-1} \cdot a}:z^{-1}a\rightarrow a$ and one outgoing 
arrow $f_a:a\rightarrow za$. 
  
In both cases, any morphism in $\C_Q$
can be uniquely expressed as 
$$
	f_{z^n \cdot a}\circ \ldots \circ f_{a},
	\quad
	a  \in Q_0,n \in \mathbb{N},
$$
to be interpreted as $i_a$ when $n=0$.
The monoidal product is given by 
\begin{equation}\label{monoidalproduct}
 	(f_{z^n \cdot a} \circ \ldots 
	\circ f_{a} ) \cdot 
	(f_{z^l \cdot b} \circ
	\ldots \circ f_{b}) =
	f_{z^{n+l} \cdot a \cdot b}
	\circ \ldots 
	\circ f_{a \cdot b}.
\end{equation} 

This product is evidently not
ULF, so we obtain:

\begin{theorem}\label{quivertheorem}
Given a Hopf quiver $Q$, the quiver Hopf algebra $kQ$ is a case of Theorem~\ref{bialgtheorem} if and only if $Q_1=\emptyset$
\end{theorem}

\begin{remark}[Another, similar, non-example]
In \cite{mc}, Crossley defines several Hopf algebra structures
on $k\langle S \rangle$ for a set $S$. In two
of them, the coalgebra structure is an incidence coalgebra (on the free monoidal category with one object and a morphism for each $s\in S$). However, just like the quiver
Hopf algebras of Cibils and Rosso, these two Hopf
algebras are not examples of our construction, as in both cases,
$\langle S \rangle \cdot \langle S
\rangle \nsubseteq
\langle S \rangle$. 
\end{remark}

\section{Weak Hopf algebras from monoidal categories}\label{weak section}
In this section we consider the following:
\begin{definition}[2-group]
A strict 2-group is a strict monoidal
groupoid in which every object is invertible,
i.e.~in which we have 
$\forall x\in Ob\C\exists y\in Ob\C : x\cdot y=1=y\cdot x$. 
\end{definition}
\begin{lemma}\label{monoidal inverses}
Let $\C$ be a 2-group. Then for any $f\in\C$
there exists $\bar f \in\C$ such that $\bar
f\cdot f=i_1=f\cdot \bar f$.  
\end{lemma}
\begin{proof}
Take $\bar f=i_{\overline{t(f)}}\cdot
f^{-1}\cdot i_{\overline{s(f)}}$, where
$\overline{t(f)},\overline{s(f)}$ denote the inverses in $(Ob\C,\cdot)$.
\end{proof}
Strict 2-groups are the objects of a
subcategory of $\mathbf{Grpd}$, where we keep
only the functors of groupoids which preserve
the group structure of $\C$ with respect to
$\cdot$. 
\begin{definition}[Source subgroup]
In any 2-group $(\C,\cdot,1)$ we define the source subgroup $$S:=\bigcup_{x\in Ob\C}\C(1,x),$$ which contains all morphisms with source $1$. $(S,\cdot)$ is a normal subgroup of $(\C,\cdot)$. 
\end{definition}

\begin{lemma}\label{2group LF lemma}
A 2-group $(\C,\cdot,1)$ is locally finite if
and only if $|S|$ is finite. In this case,
the monoidal product has the $|S|$-LF property (Definition~\ref{ULF def}). 
\end{lemma}
\begin{proof}
Multiplication by $i_x$ defines a bijection $\C(1,y)\rightarrow \C(x,xy)$ for all $x,y\in Ob\C$. As $\C$ is a groupoid, we have 
\begin{align*}
N_2(f)&=\{(f\circ g^{-1},g)\mid s(g)=s(f)\} \\
\Rightarrow |N_2(f)|&=|\{g\mid s(g)=s(f)\}|=|S|
\end{align*}
 for all $f\in\C$, and the map $
N_2(f)\times N_2(g)\rightarrow N_2(f\cdot g)$ from Definition \ref{ULF def}
has inverse 
\begin{equation*}
N_2(f\cdot g)\ni (x,y)\mapsto \{((x,y),(\bar
x\cdot a,\bar y\cdot b))\mid (x,y)\in N_2(f)
\}.\qedhere
\end{equation*}
\end{proof}
\begin{remark}
The second part of this lemma implies that a 2-group $\C$ is a M\"obius category if and only if $\C=Id\C$.
\end{remark}
The $|S|$-LF property is sufficient to define
compatible algebra and coalgebra structures
on $k\C$ using the scaled version of the
incidence coalgebra (Definition~\ref{modified  path coalg defn}). 
As it is less well-known than the defintion of of a Hopf algebra, we recall here the definition of a weak Hopf algebra, following \cite{boehmnillszlachanyi}:
\begin{definition}[Weak Hopf algebra]\label{whadef}
Let $k$ be a field. A weak $k$-Hopf algebra is a tuple $(A,\mu,1,\Delta,\epsilon,S)$ satisfying:
\begin{itemize}
\item[(A1)] $(A,\mu,1)$ is an associative unital $k$-algebra: 
\item[(A2)] $(A,\Delta,\epsilon)$ is a coassociative counital $k$-coalgebra
\item[(A3)] The coproduct is multiplicative: $\Delta(ab)=\Delta(a)\Delta(b)$
\item[(A4)] The counit is weakly multiplicative: $$\epsilon(ab)=\epsilon(a1_{(1)})\epsilon(1_{(2)}b)=\epsilon(a1_{(2)})\epsilon(1_{(1)}b) $$ 
\item[(A5)] The unit is weakly comultiplicative: $$(\Delta(1)\otimes 1)(1\otimes \Delta(1))=\Delta^{2}(1)=(1\otimes \Delta(1))(\Delta(1)\otimes 1) $$
\item[(A6)] $S:A\rightarrow A$ is a $k$-linear map (called the antipode) which satisfies\begin{align*}
S(a)&=S(a_{(1)})a_{(2)}S(a_{(3)})\\
S(a_{(1)})a_{(2)}&=\epsilon(1_{(1)}a)1_{(2)}\\
a_{(1)}S(a_{(2)})&=1_{(2)}\epsilon(a1_{(2)}).
\end{align*}
\end{itemize}
\end{definition}
\begin{remark}
Under the assumption of axioms (A1), (A2) and (A3), the version of axiom (A4) given here is equivalent to the formulation in \cite{boehmnillszlachanyi}: \begin{equation*}\label{alta4}
\epsilon(abc)=\epsilon(ab_{(1)})\epsilon(b_{(2)}c)=\epsilon(ab_{(2)})\epsilon(b_{(1)}c) 
\end{equation*}
\end{remark}
We can now state the central theorem of this section:
\begin{theorem}\label{weak hopf theorem}
Let $(\C,\cdot,1)$ be a locally finite
2-group, $k$ a field of characteristic zero
and $k\C$ the free $k$-module on $\C$. We
understand $\otimes$ to be the tensor product
$\otimes_k$ and $(k\C,\Delta,\varepsilon)$
the scaled incidence coalgebra of $k\C$ with
$\lambda =|S|$ (Definition~\ref{modified
path coalg defn}). Then
$(k\C,\cdot,i_1,\Delta,\epsilon,S)$ is a weak
Hopf algebra with antipode $S(f)=\bar f^{-1}$.
\end{theorem}
\begin{proof}
(A1) and (A2) are immediate by Theorems \ref{modified  path coalg defn} and \ref{alg thm}. (A3) then follows by Lemma \ref{2group LF lemma}, (A4) by Lemma \ref{monoidal inverses}. For (A5) we observe that 
\begin{align*}
N_2(i_1)\times N_2(i_1)\rightarrow N_3(i_1)\\
((a,b);(c,d))\mapsto (a,b\cdot c,d)
\end{align*}
defines a bijection and that $b\circ c=b\cdot
c=c\cdot b$ for $s(b)=t(c)=1$. (A6) can be
verified by a short computation.
\end{proof}

\subsection{Example: Relations on groups}\label{ex relationsweak}
Here we present the analogy to Example~\ref{monex}: let $\sim$ be a reflexive transitive relation on a group $G$ which is compatible with the multiplication as in (\ref{precrel}) and consider the category $\C$ with $Ob\C=G$, $\C(h,g)=\{(h,g)\mid g\sim h\}$. This category is a locally finite 2-group iff $\sim$ is an equivalence relation and the equivalence classes are finite. 

By Theorem~\ref{weak hopf theorem}, $k\C$ admits a weak incidence Hopf algebra structure with a coproduct which runs through equivalence classes:
$$\Delta(h,g)=\sum_{g\sim f\sim h}(h,f)\otimes (f,g). $$ In the case that $\sim$ is equality, the weak Hopf algebra obtained is the group algebra $kG$.
This can be alternatively stated in the
following way: let $G$ be a group and
$N\triangleleft G$ a normal subgroup. Define
a 2-group $\C$ by $Ob\C=G$ and
$$|\C(g,h)|=\begin{cases}1,\;\exists z\in
N:h=zg\\0,\text{ else} \end{cases}.$$ Then $k\C$ admits a weak incidence Hopf algebra if and only if $|N|$ is finite.

\subsection{Example: Automorphism 2-groups}\label{ex automorphism groupoid}
Let $\C=(Ob\C,Mor_1\C,Mor_2\C)$ be a strict
2-category. Then for each $x \in Ob\C$ there exists a strict
2-group $\text{AUT}(x)$ whose objects are the
automorphisms of $x$ as an object of $\C$ and whose morphisms are the 2-isomorphisms between these. The product of two objects $f,f'$ in $\text{AUT}(x)$ is given by their composition in $\C$ and the unit object is $i_x$. The composition of morphisms in $\text{AUT}(x)$ is given by the vertical composition of 2-morphisms in $\C$ and the product of morphisms in $\text{AUT}(x)$ is given by the horizontal composition in $\C$. According to Theorem~\ref{weak hopf theorem}, $k\text{AUT}(x)$ admits a weak incidence Hopf algebra structure iff $\{f\in Mor_1\C\mid f\sim i_x\}$ is finite. 

This example overlaps with Example~\ref{ex relationsweak}. 
Consider the 2-category $\mathbf{Grp}$, whose
objects are groups, 1-morphisms are group
homomorphisms and 2-morphisms are given by
inner automorphisms in the target group, i.e.
for $f_1,f_2:G\rightarrow H$, there exists
$\phi_h:f_1\rightarrow f_2$ iff there exists
$h\in H$ such that $f_1=f_h\circ f_2$, where
$f_h$ is the morphism in $\text{Inn}(H)$
given by conjugation by $h$. If $G$ is an
object in $\mathbf{Grp}$,
$Ob\text{AUT}(G):=\text{Aut}(G)$, the usual
group of automorphisms of $G$, and
$\text{AUT}(G)$ (the set of morphisms) contains an arrow between any pair of autmorphisms which are related by an inner automorphism of $G$. Hence $k\text{AUT}(G)$ admits a weak Hopf algebra structure iff $\text{Inn}(G)$ is finite.

\subsection{Crossed modules}
Here we will reformulate the result of Theorem~\ref{weak hopf theorem} in the language of crossed modules. 
\begin{definition}[Crossed module]
A crossed module $(G,H,\tau,\alpha)$ consists of:
\begin{itemize}
\item Groups $G,H$
\item A group action $\alpha:G\times H\rightarrow H$
\item A group homomorphism $\tau:H\rightarrow G$.
\end{itemize}
such that $\alpha$ and $\tau$ satisfy
\begin{align*}
\tau(\alpha(g,h))&=g\tau(h)g^{-1}\\
\alpha(\tau(h),h')&=hh'h^{-1}.
\end{align*}
\end{definition}
Crossed modules form the objects of a
category whose morphisms
$(f_1,f_2):(G,H,\tau,\alpha)\rightarrow
(G',H',\tau',\alpha')$ consist of pairs of group homomorphisms $f_1:G\rightarrow G'$ and $f_2:H\rightarrow H'$ such that $\tau'f_2=f_1\tau$.
\begin{theorem}[Brown and Spencer]\label{thm 2 group= crossed module}
The categories of of strict 2-groups and crossed modules are equivalent.
\end{theorem}
For the original statement and proof see \cite{brownspencer}, for further exploration see e.g. \cite{baezlauda,fb}.

Given a 2-group $(\C,\cdot,1)$, the
corresponding crossed module
$(G,H,\alpha,\tau)$ is defined as follows:
set $G=Ob\C$, $H=S$, $\alpha(g,h)=i_g\cdot h
\cdot i_{g^{-1}}$ and $\tau(h)=t(h)$.

Conversely, given a crossed module
$(G,H,\alpha,\tau)$, the corresponding strict
2-group $(\C,\cdot,1)$ has $Ob\C=G$, $\C=H\rtimes_\alpha G$, $s(h,g)=g$ and $t(h,g)=\tau(h)g$. Composition of morphisms is given by $(h',\tau(h)g)\circ (h,g)=(h'h,g)$.

 This equivalence gives us the following:
\begin{corollary}[to Theorem~\ref{weak hopf theorem}]
Let $(G,H,\tau,\alpha)$ be a crossed module,
$k$ a field and $k(H\times G)$ the free
$k$-module on the set $H\times G$. Then
$k(H\times G)$ admits the following weak Hopf algebra structure iff $|H|$ is finite:
\begin{align*}
(h,g)\cdot (h',g')&=(h\alpha(g,h'),gg')\\
\Delta(h,g)&=\sum_{h'h''=h}(h'',\tau(h')g)\otimes (h',g)\\
%\hat{(h,g)}&=(\alpha(g^{-1},h^{-1}),g^{-1})\\
S(h,g)&=\left(\alpha(g^{-1},h^{-1})^{-1},g^{-1} \right)
\end{align*}
\end{corollary}

If $\tau:H\rightarrow G$ is injective, we are
in the case of Example~\ref{ex relationsweak}, i.e. $H$ is isomorphic to a normal subgroup of $G$.


\begin{thebibliography}{9}
\bibitem{baezlauda}
  John C. Baez and Aaron D. Lauda,
  \textit{Higher-Dimensional Algebra V: 2-Groups},
Theory and Applications of Categories 12 (2004), 423-491.
\bibitem{bk}
  Christoph Bergbauer and Dirk Kreimer,
  \textit{Hopf algebras in renormalisation theory: locality
and Dyson-Schwinger equations from Hochschild cohomology},
IRMA Lect.Math.Theor.Phys. 10 (2006) 133-164.
\bibitem{bmsw} Nantel Bergeron, Stefan Mykytiuk, Frank Sottile,  Stephanie van Willigenburg,
\textit{Non-commutative Pieri operators on posets},
J. Combin. Th. Ser. A, Vol. 91, No. 1/2, Aug 2000, pp. 84-110. 
\bibitem{boehmnillszlachanyi}
Gabriella B\"ohm, Florian Nill, Korn\'el Szlach\'anyi
\textit{Weak Hopf Algebras I. Integral Theory and $C^{*}$-structure}
	arXiv:math/9805116
	
	\bibitem{brownspencer}
R. Brown and C. B. Spencer,
\textit{ G-groupoids, crossed modules, and the classifying
space of a topological group}
 Proc. Kon. Akad. v. Wet. 79 (1976), 296–302	
\bibitem{cs}
 Michele Sevegnani and Muffy Calder 
 \textit{Bigraphs with sharing},
 TR-2010-310, DCS Technical Report Series, Department of Computing Science, University of Glasgow, 2010 

\bibitem{cr}
  Claude Cibils and Marc Rosso,
  \textit{Hopf quivers},
2000 	arXiv:math/0009106.
\bibitem{clw}
Claude Cibils, Aaron Lauve and Sarah Witherspoon,
\textit{Hopf quivers and Nichols algebras in positive characteristic}
2009
\bibitem{acdk}
Alain Connes and Dirk Kreimer,
\textit{Hopf Algebras, Renormalization and Noncommutative Geometry}
Commun.Math.Phys. 199 (1998) 203-242.

\bibitem{mc}
  M.D. Crossley,
  \textit{Some Hopf Algebras of Words},
  Glasgow Math. J. 48 (2006) 575–582.
\bibitem{lf}
  Lo\"ic Foissy,
  \textit{An introduction to Hopf algebras of trees},
preprint.
\bibitem{fb}
Magnus Forrester-Barker,
 \textit{Group Objects and Internal Categories},
2002 arXiv:math/0212065v1.

\bibitem{gs}
E. L. Green and \O. Solberg,
\textit{Basic Hopf algebras and quantum groups},
Mathematische Zeitschrift 229(1):45-76 · September 1998
\bibitem{kgt}
 Imma G\`alvez-Carillo, Joachim Kock and Andrew Tonks,
 \textit{Decomposition Spaces in Combinatorics}
2016  arXiv:1612.09225v2
\bibitem{kgt3}
 Imma G\`alvez-Carillo, Joachim Kock and Andrew Tonks,
 \textit{Decomposition Spaces, incidence algebras and M\"obius inversion III: the decomposition space of M\"obius intervals. }
2015  arXiv:1512.07580
\bibitem{gr}
Darij Grinberg and Victor Reiner. 
\textit{Hopf algebras in combinatorics.}
2014 arXiv:1409.8356

\bibitem{holt} Ralf Holtkamp, 
\textit{On Hopf algebra structures over free
operads.} Adv. Math. 207 (2006), no. 2, 544–565.

\bibitem{ht}
  Hua-Lin Huang and Blas Torrecillas,
  \textit{Quiver bialgebras and monoidal categories},
Colloq. Math. 131 (2013), 287-300.
\bibitem{jm}
Ole H\o gh Jensen and Robin Milner,
\textit{Bigraphs and mobile processes (revised)},
University of Cambridge computer laboratory technical report, 2004.
\bibitem{jr}
S. A. Joni and G.-C. Rota,
\textit{Coalgebras and bialgebras in combinatorics},
Stud. Appl. Math., 61 (1979), pp. 93–139.
\bibitem{joyalstreet}
Andr\'e Joyal and Ross Street
\textit{Braided monoidal categories}
Macquarie Mathematics
Report No. 860081, November 1986.
\bibitem{kaufmannward}
Ralph M. Kaufmann, Benjamin C. Ward
\textit{Feyman Categories}
 arXiv:1312.1269v3
\bibitem{kockdse}
Joachim Kock,
\textit{Polynomial functors and combinatorial Dyson-Schwinger equations},
J. Math. Phys. 58 (2017), 041703.
\bibitem{lm}
F. W. Lawvere and M. Menni,
\textit{The Hopf Algebra of M\"obius Intervals}
Theory and Applications of Categories, Vol. 24, No. 10, 2010, pp. 221–265.
\bibitem{le}
P. Leroux
\textit{Les Categories de M\"obius}
Cahiers de Topologie et G\`eometrie Diff\`erentielle Cat\`egoriques, 16(3):280–282, 1975.
\bibitem{lr}
  Jean-Louis Loday and Mar\'ia Ronco,
  \textit{Combinatorial Hopf algebras},
2008 	arXiv:0810.0435.
\bibitem{dm}
Dominique Manchon,
\textit{Hopf algebras, from basics to applications to renormalization},
Comptes Rendus des Rencontres Mathematiques de Glanon 2001.
\bibitem{rm1} 
 Robin Milner,
 \textit{The space and motion of communicating agents},
Cambridge University Press 2008.

%\bibitem{rm2}
%Robin Milner,
%\textit{Bigraphs and Their Algebra},
%Electronic Notes in Theoretical Computer Science 209 (2008).
\bibitem{der}
  David E. Radford,
  \textit{Hopf algebras},
World Scientific, 2012.
\bibitem{MSc}
Lucia Rotheray,
\textit{Hopf Subalgebras from Green's Functions}
MSc thesis, Humboldt-Universit\"at zu Berlin, 2015.

\bibitem{ky}
Karen Yeats,
\textit{A Hopf algebraic approach to Schur function identities}
 arXiv:1511.06337v4
\end{thebibliography}
\end{document}